\documentclass{article}
\usepackage{stmaryrd,graphicx}
\usepackage{amssymb,mathrsfs,amsmath,amscd,amsthm}
\usepackage{float,color}
\usepackage{graphicx,float}
\usepackage[all,cmtip]{xy}
\usepackage{amsfonts,txfonts,pxfonts,latexsym,wasysym}

\DeclareMathAlphabet{\mathpzc}{OT1}{pzc}{m}{it}

\newcommand{\piqtop}{\pi_{1}^{qtop}}
\newcommand{\piqtopy}{\pi_{1}^{qtop}(Y,y_0)}
\newcommand{\piqtopx}{\pi_{1}^{qtop}(X,x_0)}

\newcommand{\spu}{\spanier(\mathscr{U},x_0)}

\newcommand{\spanier}{\pi^{s}}

\newcommand{\pionex}{\pi_{1}(X,x_0)}
\newcommand{\shape}{\check{\pi}_{1}(X,x_0)}

\newcommand{\scru}{\mathscr{U}}
\newcommand{\scrv}{\mathscr{V}}

\newcommand{\nerveu}{N(\mathscr{U})}
\newcommand{\nervev}{N(\mathscr{V})}

\newcommand{\pxxo}{\mathcal{P}(X,x_0)}
\newcommand{\ui}{[0,1]}

\newcommand{\cg}{\mathbf{CG}}

\newtheorem{theorem}{Theorem}
\newtheorem{lemma}[theorem]{Lemma}
\newtheorem{proposition}[theorem]{Proposition}
\newtheorem{corollary}[theorem]{Corollary}
\newtheorem{definition}[theorem]{Definition}

\newtheorem{example}[theorem]{Example}

\newtheorem{problem}[theorem]{Problem}
\newtheorem{remark}[theorem]{Remark}

\begin{document}

\title{On fundamental groups with the quotient topology}
\author{Jeremy Brazas and Paul Fabel}
\maketitle

\begin{abstract}
The quasitopological fundamental group $\pi_{1}^{qtop}(X,x_0)$ is the fundamental group endowed with the natural quotient topology inherited from the space of based loops and is typically non-discrete when $X$ does not admit a traditional universal cover. This topologized fundamental group is an invariant of homotopy type which has the ability to distinguish weakly homotopy equivalent and shape equivalent spaces. In this paper, we clarify various relationships among topological properties of the group $\pi _{1}^{qtop}(X,x_0)$ and properties of the underlying space $X$ such as `$\pi_{1}$-shape injectivity' and `homotopically path-Hausdorff.'

A space $X$ is $\pi_1$-shape injective if the fundamental group canonically embeds in the first shape group so that the elements of $\pi_1(X,x_0)$ can be represented as sequences in an inverse limit. We show a locally path connected metric space $X$ is $\pi_1$-shape injective if and only if $\pi_{1}^{qtop}(X,x_0)$ is invariantly separated in the sense that the intersection of all open invariant (i.e. normal) subgroups is the trivial subgroup. In the case that $X$ is not $\pi_1$-shape injective, the homotopically path-Hausdorff property is useful for distinguishing homotopy classes of loops and guarantees the existence of certain generalized covering maps. We show that a locally path connected space $X$ is homotopically path-Hausdorff if and only if $\pi_{1}^{qtop}(X,x_0)$ satisfies the $T_1$ separation axiom.
\end{abstract}

\section{Introduction}

This paper concerns the topology and algebraic topology of locally
complicated spaces $X$, which are not guaranteed to be locally path
connected or semilocally simply connected, and for which the familiar
universal cover is not guaranteed to exist.

The central object of study is the usual fundamental group $\pionex$, endowed with a natural quotient topology inherited from the space of based loops in $X.$ So equipped, $\pionex$ becomes a quasitopological group denoted $\piqtopx.$ Examples \cite{Brazfretopgrp} illustrate that $\piqtopx$ need not be a topological group even if $X$ is a compact metric space \cite{Fabhe}\cite{Fabcgqtop}.

A based map $f:X\to Y$, $f(x_0)=y_0$ of spaces induces a continuous group homomorphism $f_{\ast}:\piqtopx \to \piqtopy$ \cite[Proposition 3.3]{Biss}. Since homotopic maps induced the identical homomorphism on fundamental groups, it follows that $\piqtopx$ is invariant under the homotopy type of $X$. In particular, topological properties of $\piqtopx$ (for example, separation axioms and the success or failure of continuity of multiplication) have the capacity to distinguish homotopy type when standard application of weak homotopy type or shape theory fails to do so (See Example \ref{lasso} below and \cite{Brazfretopgrp}\cite{Fabe06}). Additionally, quasitopological fundamental groups are vertex groups of so-called fundamental $\mathbf{qTop}$-groupoids, which are a key tool in a recent proof of a general Nielsen-Schreier Theorem for topological groups \cite{Brazns}. 

To promote the relevance and utility of $\piqtopx$, we establish and clarify a variety of relationships among general topological properties of the space $\piqtopx$ and properties of the underlying space $X$
such as `homotopically path-Hausdorff,' and `$\pi _{1}$-shape injectivity'. For example, if $X$ is a compact metrizable space, the following facts are obtained directly from (typically more general) theorems proved in the paper at hand.
\begin{enumerate}
\item Theorem \ref{t1ishompathhaus}: If $X$ is locally path connected then 
\[\xymatrix{
X\text{ is homotopically path-Hausdorff} \ar@{<=>}[r] & \piqtopx\text{ is }T_1.
}\]
It is apparently an open question if this is equivalent to ``$\piqtopx$ is $T_{2}$." A candidate for a Peano continuum with $T_1$ but non-$T_2$ fundamental group is the so-called sombrero space from \cite{CMRZZ08} and \cite{FRVZ11}.
\item We say a subgroup of a group is an \textit{invariant} subgroup if it is closed under conjugation. Throughout this paper, we choose to use ``invariant subgroup" rather than ``normal subgroup" to avoid confusion with the separation axiom. A quasitopological group $G$ is \textit{invariantly separated} if distinct elements of $G$ can be separated by cosets of some open invariant subgroup of $G$ (equivalently, the intersection of all open invariant subgroups in $G$ is the trivial subgroup). Theorem \ref{injectiveisinvsep}: If $X$ is locally path connected, then 
\[\xymatrix{
\piqtopx\text{ is invariantly separated} \ar@{<=>}[r] & X \text{ is }\pi_1\text{-shape injective.}
}\]
\item Corollary \ref{nestedpolyhedra}: If $X$ is both locally path connected and the inverse limit of nested retracts of polyhedra, then 
\[\xymatrix{
\piqtopx\text{ is }T_1 \ar@{<=>}[r] & \piqtopx\text{ is }T_2 \ar@{<=>}[r] & X \text{ is }\pi_1\text{-shape injective.} 
}\]
\item Theorem \ref{t3tot4}: If $\piqtopx$ is $T_{3}$, then $\piqtopx$ is $T_{4}$.
\item Theorems \ref{cggroups}: If $\piqtopx\times \piqtopx$ is compactly generated, then $\piqtopx$ is a normal topological group. As a special case, if $\piqtopx$ is $T_2$ and the inductive limit of nested compact subspaces (i.e. a $k_{\omega}$-group), then $\piqtopx$ is a $T_4$ topological group.
\end{enumerate}

We leave the reader with a fundamental open question (Problem \ref{openq1}). Suppose $X$ is a compact metric space such that $\piqtopx$ is $T_{1}$. Must $\piqtopx$ be $T_{4}$?

\section{The quasitopological fundamental group}

Throughout this paper, $X$ is a path connected topological space and $x_0\in
X$ is a given basepoint. Let $\mathcal{P}(X)$ be the space of paths $
[0,1]\to X$ with the compact-open topology. A subbasis for this topology
consists of sets $\langle K,U\rangle=\{f|f(K)\subseteq U\}$ where $
K\subseteq [0,1]$ is compact and $U\subseteq X$ is open. A convenient basis for the topology of $\mathcal{P}(X)$ consists of sets of the form $\bigcap_{i=1}^{n}\langle [t_{i-1},t_i],U_i\rangle$
where $0=t_0<t_1<...<t_n=1$ and each $U_i\subset X$ is open. Additionally, the subcollection of such neighborhoods which also satisfy $t_j=\frac{j}{n}$ gives a basis for the compact-open topology. It is well-known that the compact-open topology of $\mathcal{P}(X)$ agrees with the topology of uniform convergence when $X$ is a metric space.

Let $\mathcal{P}(X,x_0)$ and $\Omega(X,x_0)$ be the subspaces of paths starting at $
x_0$ and loops based at $x_0$ respectively. Given paths $\alpha,\beta\in \mathcal{P}(X)$, $\alpha^{-}(t)=\alpha(1-t)$ denotes the reverse path and if $\beta(0)=\alpha(1)$, then $\alpha\cdot\beta$ denotes the usual concatenation of paths. We denote the constant path at $x\in X$ by $c_x$.

\begin{definition}
\emph{The \textbf{quasitopological fundamental group} of $(X,x_0)$ is the fundamental group endowed with the quotient topology induced by the canonical map $\pi:\Omega(X,x_0)\to\pi_1(X,x_0)$, $\pi(\alpha)=[\alpha]$ identifying based homotopy classes of loops. We denote it by $\pi_{1}^{qtop}(X,x_0)$.}
\end{definition}

Since homotopy classes of loops are precisely the path components of $\Omega(X,x_0)$, the group $\pi_{1}^{qtop}(X,x_0)$ is also the path component
space of $\Omega(X,x_0)$, that is, the quotient space obtained by collapsing path
components to points.

Recall that a \textbf{quasitopological group} is a group $G$ with topology
such that inversion $G\to G$, $g\mapsto g^{-1}$ is continuous and
multiplication $G\times G\to G$, $(a,b)\mapsto ab$ is continuous in each
variable. The second condition is equivalent to the condition that all right and
left translations in $G$ be homeomorphisms. For more on the general theory
of quasitopological groups, we refer the reader to \cite{AT08}.

The following lemma, first formulated in \cite{Brazfretopgrp}, brings together results from \cite{Biss} and \cite{Calcut}.

\begin{lemma} For any space $(X,x_0)$, $\pi_{1}^{qtop}(X,x_0)$ is a
quasitopological group. Moreover, a map $f:X\to Y$, $f(x_0)=y_0$ induces a
continuous homomorphism $f_{\ast}:\piqtop(X,x_0)\to \piqtop(Y,y_0)$.
\end{lemma}

Thus $\pi_{1}^{qtop}$ becomes a functor from the category of based
topological spaces (and an invariant of homotopy type) to the category of quasitopological groups and continuous homomorphisms. The isomorphism class of $\piqtopx$ does not depend on the choice of basepoint; if $\alpha:\ui\to X$ is a path, then $\piqtop(X,\alpha(1))\to \piqtop(X,\alpha(0))$, $[\gamma]\mapsto [\alpha\cdot \gamma\cdot \alpha^{-}]$ is an isomorphism of quasitopological groups \cite{Biss}.

It is also possible to consider the higher homotopy groups $
\pi_n(X,x_0)$, $n>1$ as quasitopological abelian groups in a similar fashion \cite{GHMM08}\cite{GHMM10}. Higher quasitopological homotopy groups can also fail to be topological groups \cite{Fabcgqtop}. Many of the results in this paper have analogues for these higher homotopy groups, however, we do not address them directly.

The group $\pi_{1}^{qtop}(X,x_0)$ is a discrete topological group if and only if
every loop $\alpha$ admits a neighborhood in $\Omega(X,x_0)$ which contains
only loops homotopic to $\alpha$. A more practical characterization is the following.

\begin{theorem}\cite{Calcut}\label{discrete}
If $\piqtopx$ is discrete, then $X$ is semilocally simply connected. If $X$ is locally path connected and semilocally simply connected, then $\piqtopx$ is discrete.
\end{theorem}

In particular, if $X$ has the homotopy type of a
CW-complex or a manifold, then $\pi_{1}^{qtop}(X,x_0)$ is discrete.

At the other extreme, there are many examples of (even compact) metric
spaces for which $\pi_{1}^{qtop}(X,x_0)$ is non-trivial and carries the
indiscrete topology \cite{Brazfretopgrp}\cite{Wilkins}.

\section{Separation axioms in $\pi_{1}^{qtop}(X,x_0)$}

Traditionally, shape theory has been used to study fundamental groups of locally complicated spaces. A space $X$ is said to be $\pi_1$-shape injective if the canonical homomorphism $\psi:\pionex\to \check{\pi}_1(X,x_0)$ from the fundamental group to the first shape group is injective (See Section 3.2 for further discussion). If $\psi$ is injective, then $\pionex$ can be understood as a subgroup of $\check{\pi}_1(X,x_0)$, which is an inverse limit of discrete groups. If $\psi$ is not injective, then shape theory fails to distinguish some elements of $\pionex$. When this failure occurs, there is still hope that elements of $\pionex$ which are indistinguishable by shape, can be distinguished by the quotient topology of $\piqtopx$. This possibility is one motivation for considering separation axioms in quasitopological fundamental groups.

It is well-known that every $T_0$ topological group is Tychonoff. Since $\pi_{1}^{qtop}(X,x_0)$ need not be a topological group, it is not immediately clear if one can promote separation axioms within $\piqtopx$ in a similar fashion. In this section, we relate properties of the topological space $X$ and separation axioms in $\piqtopx$. We also identify some cases where it is possible to strengthen a given separation axiom.

\subsection{On $T_0$ and $T_1$}

The following general facts about quasitopological groups are
particularly useful in understanding lower separation axioms in $\pi_{1}^{qtop}(X,x_0)$. If $G$ is a space and $g\in G$, then $\overline{g}$ denotes the closure of the singleton $\{g\}$. Recall that in a space $G$, distinct points $g,h\in G$ are \textbf{topologically indistinguishable} if every neighborhood of $g$ contains $h$ (i.e. $g\in \overline{h}$) and every neighborhood of $h$ contains $g$ (i.e. $h\in \overline{g}$).
\begin{lemma}\label{topindist}
If $G$ is a quasitopological group and $g\in G$, then $\overline{g}$ is precisely the set of elements which are topologically indistinguishable from $g$.
\end{lemma}
\begin{proof}
By definition, if $h$ is topologically indistinguishable from $g$, then $h\in \overline{g}$. Suppose $h\in \overline{g}$ with $h\neq g$ and that $U$ is any neighborhood of $g$. It suffices to show $h\in U$. Since $G$ is a quasitopological group, $g^{-1}U$ is a neighborhood of the identity element $e$ and $h=he^{-1}\in h(g^{-1}U)^{-1}=hU^{-1}g$ where $hU^{-1}g$ is open. By assumption, we have $g\in hU^{-1}g$ and therefore $e\in hU^{-1}$. It follows that $h\in U$.
\end{proof}
Note the proof of the previous lemma makes use of the continuity of inversion and the translations in $G$. Lemma \ref{topindist} does not hold when $G$ is replaced by any $T_0$ non-$T_1$ topological space. 

According to Lemma \ref{topindist}, if $G$ is a quasitopological group and $g\in G$, then $\overline{g}$ is contained in every neighborhood of $g$. It is well-known that if $H$ is a subgroup of quasitopological group $G$, then the closure $\overline{H}$ is a subgroup of $G$ \cite[1.4.13]{AT08}. A straightforward argument gives that if $H$ is an invariant subgroup of $G$, then $\overline{H}$ is also an invariant subgroup of $G$. Combining these observations, we obtain the following Corollary.
\begin{corollary}
Suppose $G$ is a quasitopological group with identity $e$. Then $\overline{e}$ is a closed invariant subgroup of $G$ contained in every open neighborhood of $e$.
\end{corollary}
\begin{corollary}
\label{tzero} If $G$ is a quasitopological group, then the following are equivalent:
\begin{enumerate}
\item $G$ is $T_0$,
\item $G$ is $T_1$,
\item The trivial subgroup is closed in $G$.
\end{enumerate}
\end{corollary}
\begin{proof}
3. $\Rightarrow$ 2. follows from the homogeneity of $G$ and 2. $\Rightarrow$ 1. is clear. 1. $\Rightarrow$ 3. If $G$ is $T_0$, then $\overline{e}=e$ by Lemma \ref{topindist}. Thus the trivial subgroup is closed in $G$.
\end{proof}
Recall the \textbf{Kolmogorov quotient} (or $T_0$-identification space) of a topological space $G$ is the quotient space $G/\sim$ where $g\sim h$ iff $g$ and $h$ are topologically indistinguishable. Every open neighborhood in $G$ is saturated with respect to the (open and closed) quotient map $q:G\to G/\sim$. In the case that $G$ is a quasitopological group, two elements $g,h$ have the same closure $\overline{g}=\overline{h}$ iff they are topologically indistinguishable. Thus the Kolmogorov quotient of $G$ is the $T_1$ quotient group $G/\overline{e}$ and the group projection $G\to G/\overline{e}$ is the identification map. 

We apply to quasitopological groups the well-known fact that many topological properties of a space translate to corresponding properties of the Kolmogorov quotient and vice versa \cite{Kur}.

\begin{lemma}
\label{topropstokolmquotient} Let $G$ be a quasitopological group with
identity $e$. Then $G$ is a topological group iff $G/\overline{e}$ is a
topological group. Additionally, $G$ is compact (lindel\"{o}f, first countable,
pseudometrizable, regular, normal, paracompact) iff $G/\overline{e}$ is
compact (resp. lindel\"{o}f, first countable, pseudometrizable, regular, normal,
paracompact).
\end{lemma}
\begin{proof}
We prove only the first statement since the additional statements are easily verifiable for Kolmogorov quotients of general topological spaces. In general, quotient groups of a topological group (with the quotient topology) are topological groups. Thus $G/\overline{e}$ is a topological group whenever $G$ is. Suppose $G/\overline{e}$ is a topological group and $q:G\to G/\overline{e}$ is the canonical quotient map. It suffices to check that multiplication $\mu:G\times G\to G$ is continuous. If $U\subseteq G$ is open, then $U=q^{-1}(V)$ for some open set $V\subseteq G/\overline{e}$. Since multiplication $\nu:G/\overline{e}\times G/\overline{e}\to G/\overline{e}$ is continuous, 
$(q\times q)^{-1}(\nu^{-1}(U))=\mu^{-1}(q^{-1}(V))=\mu^{-1}(U)$ is open in $G\times G$.
\end{proof}

We now provide a characterization of the $T_1$ axiom in fundamental groups
using a relative version of the notion of ``homotopically path-Hausdorff" introduced in \cite{FRVZ11}. See also \cite{Virk2}. 

\begin{definition}
\emph{Let $C$ be a subset of $\pi_{1}(X,x_0)$. The space $X$ is \textbf{
homotopically path-Hausdorff relative to} $C$ if for every pair of paths $\alpha,\beta\in\mathcal{P}(X,x_0)$ such that $\alpha(1)=\beta(1)$ and $[\alpha\cdot\beta^{-}]\notin C$, there is a partition $0=t_0<t_1<t_2<...<t_n=1$ and a sequence of open sets $U_1,U_2,...,U_n$ with $\alpha([t_{i-1},t_i])\subset U_i$, such that if $\gamma:[0,1]\to X$ is another path satisfying $\gamma([t_{i-1},t_i])\subset U_i$ for $1\leq i\leq n$ and $\gamma(t_i)=\alpha(t_i)$ for $0\leq i\leq n$, then $[\gamma\cdot\beta^{-}]\notin C$. We say $X$ is \textbf{homotopically path-Hausdorff} if it is homotopically path-Hausdorff relative to the trivial subgroup $C=1$. }
\end{definition}

\begin{lemma}
\label{equivalence1} Suppose $C$ is a subset of $\pi_{1}(X,x_0)$ and $C\neq \pionex$. If $C$ is closed in $\piqtopx$, then $X$ is homotopically
path-Hausdorff relative to $C$. If $X$ is locally path connected and
homotopically path-Hausdorff relative to $C$, then $C$ is closed in $\piqtopx$.
\end{lemma}

\begin{proof}
If $C$ is closed in $\piqtopx$, then $\pi^{-1}(C)$ is closed in $\Omega(X,x_0)$ since $\pi:\Omega(X,x_0)\to \piqtopx$ is continuous. Suppose $\alpha,\beta\in\pxxo$ such that $\alpha(1)=\beta(1)$ and $[\alpha\cdot\beta^{-}]\notin C$. Since $\alpha\cdot\beta^{-}\notin\pi^{-1}(C)$, there is a basic open neighborhood of the form $\mathcal{U}=\bigcap_{j=1}^{n}\left\langle \left[\frac{j-1}{n}\frac{j}{n}\right],U_j\right\rangle$ such that $\mathcal{U}\cap \pi^{-1}(C)=\emptyset$. We may assume $n$ is even. Now consider the partition given by $t_j=\frac{2j}{n}$ for $0\leq j\leq n/2$. We have $\alpha([t_{j-1},t_j])\subseteq U_j$ for $1\leq j\leq n/2$. Suppose $\gamma\in \pxxo$ is a path such that $\gamma([t_{j-1},t_j])\subseteq U_j$ for $1\leq j\leq n/2$ and $\gamma(t_j)=\alpha(t_j)$ for $0\leq j\leq n/2$. Clearly $\gamma\cdot \beta^{-}\in \mathcal{U}$ and thus $[\gamma\cdot\beta^{-}]\notin C$. Therefore $X$ is homotopy path-Hausdorff relative to $C$.

Suppose $X$ is locally path connected and homotopically path-Hausdorff relative to $C$. Since $\pi:\Omega(X,x_0)\to\piqtopx$ is quotient, it suffices to show $\pi^{-1}(C)$ is closed in $\Omega(X,x_0)$. Let $\alpha\in \Omega(X,x_0)$ such that $[\alpha]\notin C$ and let $\beta$ be the constant path at $x_0$ so that $[\alpha\cdot\beta^{-}]\notin C$. By assumption, there is a partition $0=t_0<t_1<t_2<...<t_n=1$ and a sequence of open sets $V_1,V_2,...,V_n$ with $\alpha([t_{i-1},t_i])\subset V_i$ such that if $\zeta\in\pxxo$ is another path satisfying $\zeta([t_{i-1},t_i])\subset V_i$ and $\zeta(t_i)=\alpha(t_i)$, then $[\zeta]=[\zeta\cdot\beta^{-}]\notin C$. Since $X$ is locally path connected, we may assume each $V_i$ is path connected. For $i=1,...,n-1$, find a path connected neighborhood $W_i$ of $\alpha(t_i)$ contained in $V_{i}\cap V_{i+1}$. Now
\[\mathcal{V}=\bigcap_{i=1}^{n-1}\langle [t_{i-1},t_i],V_i\rangle\cap \bigcap_{i=1}^{n-1}\langle \{t_i\},W_i\rangle\] is an open neighborhood of $\alpha$ in $\Omega(X,x_0)$. We claim $\mathcal{V}\cap \pi^{-1}(C)=\emptyset$. Suppose $\delta\in \mathcal{V}$. Since $\delta(t_i)\in W_i$ for $i=1,...,n-1$, find paths $\epsilon_i:\ui\to W_i$ from $\delta(t_i)$ to $\alpha(t_i)$. Let $\delta_i$ denote the path given by restricting $\delta$ to $[t_{i-1},t_i]$. Note 
\begin{eqnarray*}
\delta &\simeq& \delta_1\cdot\delta_2\cdots \delta_n \\
&\simeq& (\delta_1\cdot \epsilon_1)\cdot(\epsilon_{1}^{-}\cdot\delta_2 \cdot \epsilon_1)\cdot (\epsilon_{2}^{-}\cdot\delta_3 \cdot \epsilon_2)\cdots (\epsilon_{n-2}^{-}\cdot\delta_{n-1} \cdot \epsilon_{n-1}^{-})\cdot(\epsilon_{n-1}\cdot\delta_n)
\end{eqnarray*}
Let $\zeta_i$ be the $i$-th factor of this last concatenation and define a path $\zeta$ to be $\zeta_i$ on $[t_{i-1},t_i]$ and note $[\zeta]=[\delta]$. Since $\zeta([t_{i-1},t_i])\subseteq V_i$ and $\zeta(t_i)=\alpha(t_i)$, we have $[\zeta]\notin C$. Thus $[\delta]=[\zeta]\notin C$.
\end{proof}
The following theorem now follows directly from Lemma \ref{tzero} and Lemma \ref{equivalence1} (when $C=1$ is the trivial group).
\begin{theorem}\label{t1ishompathhaus}
A locally path connected space $X$ is homotopically path-Hausdorff iff $\piqtopx$ is $T_1$.
\end{theorem}

The connection between the $T_1$ axiom in $\piqtopx$ and the homotopically path-Hausdorff property motivates the following application of Lemma \ref{equivalence1} to the existence of generalized coverings in the sense of \cite{FZ07}.

Suppose $X$ is locally path connected and $H$ is a fixed subgroup of $\pi_{1}(X,x_0)$. Define an equivalence relation 
$\sim$ on $\mathcal{P}(X,x_0)$ by $\alpha\sim \beta$ iff $\alpha(1)=\beta(1)$ and $[\alpha\cdot\beta^{-}]\in H$ The equivalence class
of $\alpha$ is denoted $[\alpha]_{H}$. Let $\widetilde{X}_{H}=\mathcal{P}(X,x_0)/\sim$ be the set of equivalence classes with the \textbf{whisker topology} (or sometimes called the standard topology). A basis for the whisker topology consists of basic neighborhoods 
\begin{equation*}
B_H([\alpha]_{H},U)=\{[\alpha\cdot\epsilon]_{H}|\epsilon([0,1])\subseteq U\}
\end{equation*}
where $U\subseteq X$ is open and $\alpha\in \mathcal{P}(X,x_0)$. We take $\widetilde{x}_H=[c_{x_0}]_H$ to be the basepoint of $\widetilde{X}_H$. Note that
if $[\beta]_H\in B_H([\alpha]_{H},U)$, then $B_H([\alpha]_{H},U)=B_H([\beta]_{H},U)$.

Proofs of the following statements can be found in \cite{FZ07}. The endpoint projection $p_H:\widetilde{X}_{H}\to X$, $p_H([\alpha]_H)=\alpha(1)$
is known to be a continuous open surjection. Every path $\alpha:[0,1]\to X$ admits a continuous ``standard lift" $\widetilde{\alpha}^{st}:[0,1]\to \widetilde{X}_{H}$ given by $\widetilde{\alpha}
^{st}(t)=[\alpha_t]_H$ where $\alpha_t(s)=\alpha(ts)$. Recall that $p_H:\widetilde{X}_{H}\to X$ has the \textbf{unique path lifting property} if whenever $\alpha,\beta:[0,1]\to\widetilde{X}_{H}$ are paths such that $\alpha(0)=\beta(0)$ and $p_H\circ \alpha=p_H\circ \beta$, then $\alpha=\beta$. Whenever $p_H$ has the unique path lifting property, it is a generalized covering map in the sense of \cite{FZ07}.

A number of authors have identified conditions sufficient to conclude that $p_H$ has unique path lifting. For instance, if $H$ is a certain invariant subgroup (the intersection of Spanier groups or the kernel of the first shape map), then $p_H$ has unique path lifting. Unique path lifting for these same invariant subgroups follows from the approach of \cite{BDLM08}. Fisher and Zastrow show in \cite{FZ13} that if $H$ is an open (not necessarily invariant) subgroup of $\piqtopx$, then $p_H$ is a semicovering in the sense of \cite{Brazsemi} and has unique path lifting. Finally, Theorem 2.9 of \cite{FRVZ11} states that if $X$ is homotopically path-Hausdorff, then $p_H$ has the unique path lifting property when $H=1$. The following theorem systematically generalizes all of the above cases of unique path lifting. The proof is essentially the same as that of Theorem 2.9 of \cite{FRVZ11}, however, we include it for completion.

\begin{theorem}\label{closed}
If $X$ is locally path connected and $H$ is a closed subgroup in $\piqtopx$, then $p_H:\widetilde{X}_{H}\to X$ has the unique path lifting property.
\end{theorem}

\begin{proof}
According to Lemma \ref{equivalence1}, it suffices to replace the condition that $H$ is closed in $\piqtopx$ with the condition that $X$ is homotopy path-Hausdorff relative to $H$. Suppose $\alpha\in\pxxo$ is a path such that there is a lift $\beta:\ui\to \widetilde{X}_{H}$, $\beta(t)=[\beta_t]_{H}$ of $\alpha$ such that $\beta(0)=\widetilde{\alpha}^{st}(0)=\widetilde{x}_{H}$ but $\beta(1)=[\beta_1]_{H}\neq [\alpha]_{H}= \widetilde{\alpha}^{st}(1)$. We check that $X$ is not homotopically path-Hausdorff relative to $H$.

Note $[\alpha\cdot \beta_{1}^{-}]\notin H$ and consider any partition $0=t_0<t_1<\dots <t_n=1$ and path connected open sets $U_1,...,U_n$ such that $\alpha([t_{i-1},t_i])\subset U_i$. Consider a fixed value of $i$ and observe that $B_H([\beta_t]_{H},U_i)$ is an open neighborhood of $\beta(t)$ for each $t\in [t_{i-1},t_i]$. Therefore there is a subdivision $t_{i-1}=s_0<s_1<\dots <s_m=t_{i}$ such that $\beta([s_{j-1},s_j])\subseteq B([\beta_{s_{j-1}}]_{H},U_i)$ for each $j=1,...,m$. In particular, there is a path $\delta_{j}:\ui\to U_i$ from $\alpha(s_{j-1})$ to $\alpha(s_j)$ such that $[\beta_{s_{j-1}}\cdot \delta_j]_{H}=[\beta_{s_{j}}]_H$. 

The concatenation $\gamma_{i}=\delta_1\cdot\delta_2\cdots\cdot\delta_m$ is a path in $U_i$ from $\alpha(t_{i-1})$ to $\alpha(t_i)$. Since $g_j=[\beta_{s_{j-1}}\cdot \delta_j\cdot\beta_{s_{j}}^{-}]\in H$ for each $j$, we have \[h_i=[\beta_{t_{i-1}}\cdot\gamma_{i} \cdot \beta_{t_i}^{-}]=g_1g_2\cdots g_m\in H.\] Let $\gamma$ be the path defined as $\gamma_i$ on $[t_{i-1},t_i]$. Then $\gamma([t_{i-1},t_i])\subseteq U_i$ and $\gamma(t_i)=\alpha(t_i)$. Note that $[\beta_{0}\cdot\gamma\cdot \beta_{1}^{-}]=h_1h_2\cdots h_n\in H$ and $[\beta_0]\in H$. Thus $[\gamma\cdot \beta_{1}^{-}]\in H$ showing that $X$ cannot be homotopy path-Hausdorff relative to $H$.
\end{proof}
Andreas Zastrow has announced that the converse of Theorem \ref{closed} does not hold for Peano continua in the case when $H=1$ (See \cite{VZ}).
\subsection{On $T_2$ and invariantly separated groups}
In general, a $T_1$ quasitopological group need not be $T_2$ (for instance, any infinite group with the cofinite topology). A second countable $T_2$ (but non-regular) space $X$ such that $\piqtopx$ is $T_1$ but not $T_2$ is constructed in \cite[Example 4.13]{Brazfretopgrp}. 

Recall a topological space $A$ is \textbf{totally separated} if whenever $a\neq b$, there is a clopen set $U$ containing $a$ but not $b$. Equivalently, every point is the intersection of clopen sets. The following is a stronger notion for groups.

\begin{definition}
\emph{A quasitopological group $G$ is \textbf{invariantly separated} if
for every non-identity element $g$, there is an open invariant subgroup $N$
such that $g\notin N$. Equivalently, the intersection of all open invariant
subgroups in $G$ is the trivial subgroup of $G$.}
\end{definition}
\begin{proposition}
If $G$ is a quasitopological group, then
\[\xymatrix{
G\text{ is invariantly separated} \ar@{=>}[r] &G\text{ is totally separated} \ar@{=>}[r]&G\text{ is }T_2.
}\]
\end{proposition}
\begin{proof}
Note that cosets of open subgroups in $G$ are both open and closed. If $G$ is invariantly separated and $g\neq h$ in $G$, then there is an open invariant subgroup $N\subseteq G$ such that $gh^{-1}\notin N$. Then the coset $Nh$ is a clopen subset of $h$ which does not contain $g$. Thus $G$ is totally separated. It is evident from the definition that every totally separated topological space is $T_2$.
\end{proof}
\begin{proposition}\label{productinvar}
If $G_{\lambda}$ is a family of invariantly separated quasitopological groups, then the product $\prod_{\lambda}G_{\lambda}$ is invariantly separated.
\end{proposition}
\begin{proposition}\label{injection}
\label{invarseparated} If $H,G$ are quasitopological groups such that $G$ is
invariantly separated (resp. totally separated, $T_2$), and $f:H\to G$ is a continuous monomorphism, then $H$ is invariantly separated (resp. totally separated, $T_2$).
\end{proposition}

\begin{proof}
Suppose $h\in H$ is not the identity. Since $f(h)$ is not the identity in $G$, there is an open invariant subgroup $N$ in $G$ such that $f(h)\notin N$. Thus $f^{-1}(N)$ is an invariant subgroup of $H$, which is open by continuity of $f$, and $h\notin f^{-1}(N)$.

The proposition for totally separated and $T_2$ groups holds in the general category of topological spaces by standard arguments.
\end{proof}

To observe a familiar condition sufficient to know that $\piqtopx$ is invariantly separated, we recall some basic constructions from shape theory. We refer the reader to \cite{BF}\cite{MS82} for a more detailed treatment. The construction for fundamental groups of based spaces is addressed specifically in \cite{BF}.

Suppose $X$ is paracompact Hausdorff and $cov(X)$ is the directed (by refinement) set of pairs $(\scru,U_0)$ where $\scru$ is a locally finite open cover of $X$ and $U_0$ is a distinguished element of $\scru$ containing $x_0$. Given $(\scru,U_0)\in cov(X)$ let $N(\scru)$ be the abstract simplicial complex which is the nerve of $\scru$. In particular, $\scru$ is the vertex set of $U$ and the n vertices $U_1,...,U_n$ span an n-simplex iff $\bigcap_{i=1}^{n}U_i\neq \emptyset$. The geometric realization $|\nerveu|$ is a polyhedron and thus $\pi_{1}^{qtop}(|\nerveu|,U_0)$ is a discrete group.

Given a pair $(\scrv,V_0)$ which refines $(\scru,U_0)$, a simplicial map $p_{\scru\scrv}:|\nervev|\to |\nerveu|$ is constructed by sending a vertex $V\in \scrv$ to some $U\in\scru$ for which $V\subseteq U$ (in particular, $V_0$ is mapped to $U_0$) and extending linearly. The map $p_{\scru\scrv}$ is unique up to homotopy and thus induces a unique homomorphism $p_{\scru\scrv\ast}:\pi_1(|\nervev|,V_0)\to \pi_1(|\nerveu|,U_0)$. The inverse system \[(\pi_1(|\nerveu|,U_0),p_{\scru\scrv\ast},cov(X))\] of discrete groups is the \textbf{fundamental pro-group} and the limit $\check{\pi}_1(X,x_0)$ (topologized with the usual inverse limit topology) is the \textbf{first shape homotopy group}.

Given a partition of unity $\{\phi_{U}\}_{U\in \mathscr{U}}$ subordinated to $\scru$ such that $\phi_{U_0}(x_0)=1$, a map $p_{\scru}:X\to |\nerveu|$ is constructed by taking $\phi_{U}(x)$ (for $x\in U$, $U\in \scru$) to be the barycentric coordinate of $p_{\mathscr{U}}(x)$ corresponding to the vertex $U$. The induced continuous homomorphism $p_{\scru\ast}:\piqtopx\to \pi_{1}^{qtop}(|\nerveu|,U_0)$ satisfies $p_{\scru\ast}\circ p_{\scru\scrv\ast}=p_{\scrv\ast}$ whenever $(\scrv,V_0)$ refines $(\scru,U_0)$. Thus there is a canonical, continuous homomorphism $\psi:\piqtopx\to \shape$.
\begin{definition}\emph{
The space $X$ is $\mathbf{\pi_1}$\textbf{-shape injective} if $\psi:\pionex\to \shape$ is a monomorphism.}
\end{definition}

\begin{example}\label{t1butnotshapeinj}
\emph{
The authors of \cite{FRVZ11} construct a non-$\pi_1$-shape injective Peano continuum $Y'$ and provide a proof sketch that $Y'$ is homotopically path-Hausdorff (See \cite[Theorem 3.7]{FRVZ11}). Combining Theorem \ref{t1ishompathhaus} with this observation implies the existence of a non-$\pi_1$-shape injective Peano continuum $X$ for which $\piqtopx$ is $T_1$. In such a case, the elements of $\piqtopx$ cannot be distinguished by shape but can be distinguish topologically in $\piqtopx$.
}
\end{example}

\begin{proposition}\label{injectivetoinvsep}
If $X$ is $\pi_1$-shape injective, then $\piqtopx$ is invariantly separated.
\end{proposition}
\begin{proof}
For $\scru\in cov(X)$, the group $G_{\scru}=\pi_{1}^{qtop}(|\nerveu|,U_0)$ is discrete. Every discrete group is invariantly separated and thus $G=\prod_{\scru}G_{\scru}$ is invariantly separated by Proposition  \ref{productinvar}. Since, by assumption, $\piqtopx$ continuously injects into $\shape$ and $\shape$ is a sub-topological group of $G$, we apply Proposition \ref{injection} to see that $\piqtopx$ is invariantly separated.
\end{proof}
\begin{corollary}\label{injectivetohausdorff}
If $X$ is $\pi_1$-shape injective, then $\piqtopx$ is $T_2$.
\end{corollary}

The following statement is a corollary of Proposition \ref{injectivetoinvsep} only in the sense that it follows from Theorem \ref{injectivetohausdorff} and known cases of $\pi_1$-shape injectivity. These cases are typically non-trivial to confirm. For example, if $X$ is a 1-dimensional metric space \cite{Eda98} or a subset of $\mathbb{R}^2$ \cite{FZ05}, then $X$ is $\pi_1$-shape injective.

\begin{corollary}
If $X$ is a 1-dimensional metric space or a subset of $\mathbb{R}^2$, then $\piqtopx$ is invariantly separated and therefore $T_2$.
\end{corollary}
\begin{example}
\emph{
For general compact metric spaces, the converse of Proposition \ref{injectivetoinvsep} is false. Consider the compact space $Z\subset\mathbb{R}^{3}$ of \cite{FRVZ11} obtained by rotating the closed topologist's sine curve (so the linear path component forms a cylinder) and connecting the two resulting surface components by attaching a single arc. It is easy to see that $\pi_{1}^{qtop}(Z,z_0)$ and $\check{\pi}_{1}(Z,z_0)$ are both isomorphic to the discrete group of integers (and thus invariantly separated), however, $\psi:\pi_{1}^{qtop}(Z,z_0)\to \check{\pi}_{1}(Z,z_0)$ is the trivial homomorphism.
}
\end{example}

To show the converse of Proposition \ref{injectivetoinvsep} holds for locally
path connected spaces, we recall the notion of Spanier group \cite{Spanier66}. For more related to Spanier groups see \cite{BF}\cite{FRVZ11}\cite{FZ07}\cite{MPT}\cite{Wilkins}. Our approach is closely related to that in \cite{FZ13}.

\begin{definition}
\label{spaniergroupdef}\emph{\ Let $\mathscr{U}$ be an open cover of $X$.
The \textbf{Spanier group of} $X$ \textbf{with respect to} $\mathscr{U}$ is
the subgroup $\pi^{s}(\mathscr{U},x_0)\leq \pi_{1}(X,x_0)$ generated by elements of the form $[\alpha][\gamma][\alpha^{-}]$
where $\alpha\in \pxxo$ and $\gamma:[0,1]\to U$ is a loop based at 
$\alpha(1)$ for some $U\in\mathscr{U}$. The \textbf{Spanier group} of $X$ is the intersection $\pi^{s}(X,x_0)=\bigcap_{\mathscr{U}}\pi^{s}(\mathscr{U},x_0)$ of all Spanier
groups with respect to open covers.}
\end{definition}

\begin{lemma}
If $X$ is locally path connected and $\mathscr{U}$ is an open cover of $X$,
then $\pi^{s}(\mathscr{U},x_0)$ is an open invariant subgroup of $\pi_{1}^{qtop}(X,x_0)$.
\end{lemma}

\begin{proof}
The subgroup $\pi^{s}(\mathscr{U},x_0)$ is invariant by construction. By classical covering space theory \cite{Spanier66}, there is a covering map $p:Y\to X$, $p(y_0)=x_0$ such that the image of the injection $p_{\ast}:\pi_1(Y,y_0)\to \pionex$ is precisely $\spu$. Since $p$ is a covering map, the induced homomorphism $p_{\ast}:\piqtopy\to \piqtopx$ is an open embedding of quasitopological groups \cite{Brazsemi}. In particular, $p_{\ast}(\piqtopy)=\spu$ is open in $\piqtopx$.
\end{proof}

Since open subgroups of a quasitopological group are also closed and $\pi^{s}(X,x_0)=\bigcap_{\mathscr{U}}\pi^{s}(\mathscr{U},x_0)$, we obtain the following corollary.

\begin{corollary}
If $X$ is locally path connected, then $\pi^{s}(X,x_0)$ is a closed
invariant subgroup of $\pi_{1}^{qtop}(X,x_0)$.
\end{corollary}
In general, $\pi^{s}(X,x_0)$ is not open in $\piqtopx$. In fact, when $X$ is locally path connected, $\pi^{s}(X,x_0)$ is open iff $X$ admits a categorical (but not necessarily simply connected) covering space.
\begin{lemma}
\label{totallynormalseparated} If $X$ is locally path connected, then $\piqtopx$ is invariantly separated iff $\pi^{s}(X,x_0)=1$.
\end{lemma}

\begin{proof}
Since the groups $\spu$ are open invariant subgroups of $\piqtopx$, the first direction is obvious. For the converse, suppose $G=\piqtopx$ is invariantly separated. Suppose $g\in G$ is not the identity element. Since $G$ is invariantly separated, there is an open invariant subgroup $N\subset G$ such that $g\notin N$. It suffices to find an open cover $\scru$ of $X$ such that $\spu\subseteq N$. Let $\eta:\ui\to X$ be any path starting at $x_0$. Note that $\eta\cdot\eta^{-}$ is a null-homotopic loop based at $x_0$. Since $N$ is open, $\pi^{-1}(N)$ is an open neighborhood of $\eta\cdot\eta^{-}$ in $\Omega(X,x_0)$. Find a basic neighborhood $\mathcal{V}_{\eta}=\bigcap_{j=1}^{2n}\left\langle \left[\frac{j-1}{2n},\frac{j}{2n}\right],V_j\right\rangle$ of $\eta\cdot\eta^{-}$ contained in $\pi^{-1}(N)$. Note that $\eta(1)\in V_n$. Since $X$ is locally path connected, there is a path connected neighborhood $U_{\eta}$ such that $\eta(1)\in U_{\eta}\subset V_n$. Now $\scru=\{U_{\eta}|\eta\in\pxxo\}$ is an open cover of $X$. 

Suppose $[\alpha\cdot\gamma\cdot\alpha^{-}]$ is a generator of $\spu$ where $\gamma$ is a loop with image in $U_{\eta}$ where $\eta(0)=x_0$. Let $\epsilon:\ui\to U_{\eta}$ be a path from $\eta(1)$ to $\alpha(1)=\gamma(0)=\gamma(1)$. Note that an appropriate reparameterization of $\eta\cdot\epsilon\cdot\gamma\cdot\epsilon^{-}\cdot\eta^{-}$ lies in $\mathcal{V}_{\eta}$. Thus $[\eta\cdot\epsilon][\gamma][\eta\cdot\epsilon]^{-1}\in N$. Since $N$ is invariant \[[\alpha][\gamma][\alpha^{-}]=[\alpha\cdot(\eta\cdot\epsilon)^{-}] [\eta\cdot\epsilon][\gamma][(\eta\cdot\epsilon)^{-}][\eta\cdot\epsilon\cdot\alpha^{-}]\in [\alpha\cdot(\eta\cdot\epsilon)^{-}]N\left[\left(\alpha\cdot\left(\eta\cdot\epsilon\right)^{-}\right)^{-}\right]=N\]Thus $\spu\subseteq N$, completing the proof.
\end{proof}

\begin{theorem}\label{injectiveisinvsep}
If $X$ is locally path connected, paracompact Hausdorff (for instance, if $X$
is a Peano continuum), then the following are equivalent:

\begin{enumerate}
\item $X$ is $\pi_1$-shape injective
\item $\pi^{s}(X,x_0)=1$
\item $\pi_{1}^{qtop}(X,x_0)$ invariantly separated.
\end{enumerate}
\end{theorem}

\begin{proof}
1. $\Leftrightarrow$ 2. is the main result of \cite{BF} and 2. $\Leftrightarrow$ 3. follows from Lemma \ref{totallynormalseparated}.
\end{proof}

The results above show the data of the fundamental group of a Peano continuum $X$ retained by shape (i.e. the fundamental pro-group) is retained by the topology of $\piqtopx$ in the form of open invariant subgroups. Example \ref{t1butnotshapeinj} provides a space $X$ for which shape fails to distinguish uncountably many elements non-identity elements from the identity but for which the quotient topology of $\piqtopx$ can distinguish any two distinct elements. It is apparently an open question whether or not the converse of Corollary \ref{injectivetohausdorff} holds when $X$ is locally path connected.
\begin{problem}\label{openq2}
\emph{
Is there a Peano continuum $X$ which is not $\pi_1$-shape injective but is such that $\piqtopx$ is $T_2$?
}
\end{problem}

\begin{remark}
\emph{
The property ``totally separated" may also be interpreted within shape theory. If $X$ is metrizable, then $\piqtopx$ is totally separated iff $\Omega(X,x_0)$ is $\pi_0$-shape injective in the sense that the canonical function \[\piqtopx=\pi_0(\Omega(X,x_0))\to\check{\pi}_{0}(\Omega(X,x))\] to the zeroth shape set is injective. The details are straightforward and are left to the reader.
}
\end{remark}

We conclude this section with a special case where all of the separation properties mentioned so far coincide.

\begin{definition}
\emph{ Let $X_1\subseteq X_2\subseteq \cdots$ be a nested sequence of
subspaces where $X_{j}$ is a closed retract of $X_{j+1}$, i.e. there is a
map $r_{j+1,j}:X_{j+1}\to X_j$ such that if $s_{j+1,j}:X_j\to X_{j+1}$ is
the inclusion, then $r_{j+1,j}\circ s_{j+1,j}=id_{X_j}$. The inverse limit $
X=\varprojlim (X_j,r_{j+1,j})$ is called the \textbf{inverse limit of nested
retracts}. }
\end{definition}

Suppose $X=\varprojlim (X_j,r_{j+1,j})$ is an inverse limit of nested retracts. If $j<k$, let $r_{k,j}:X_k\to X_j$ and $s_{k,j}:X_j\to X_k$ be the obvious
compositions of retractions and sections respectively. The projection maps $%
r_j:X\to X_j$ are retractions with sections $s_j:X_j\to X$ given by
\begin{equation*}
s_j(x)=(r_{j,1}(x),r_{j,2}(x),...,r_{j,j-1}(x),x,s_{j+1,j}(x),s_{j+2,j}(x),...).
\end{equation*}
This is well-defined since $r_{k+1,k}(s_{k+1,j}(x))=s_{k,j}(x)$ when $j<k$. Since we may view $X_j$ as a subspace of $X$ for each $j$, we take a
basepoint in $X$ to be a point $x_0\in X_1\subset X$.

\begin{example}\label{he1}
\emph{
Let $C_n\subset \mathbb{R}^2$ be the circle of radius $\frac{1}{n}$ centered at $\left(\frac{1}{n},0\right)$ (See Figure 1). The usual \textbf{Hawaiian earring} $\mathbb{H}=\bigcup_{n\geq 1}C_n$ is a locally path connected inverse limit of nested retracts where $X_j=\bigcup_{n=1}^{j}C_n$ and $r_{j+1,j}:X_{j+1}\to X_j$ collapses $C_{j+1}$ to the canonical basepoint. 
}
\end{example}
\begin{figure}[H]
\centering \includegraphics[height=1.5in]{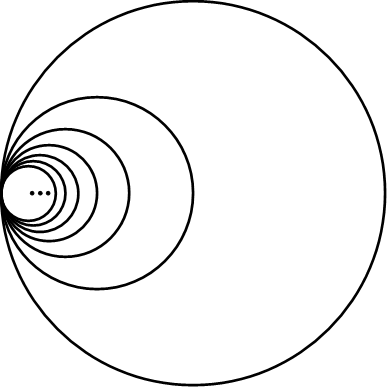}
\caption{The Hawaiian earring $\mathbb{H}$}
\end{figure}

The projection maps $r_j:X\to X_j$ induce continuous homomorphisms $(r_{j})_{\ast}:
\pi_{1}^{qtop}(X,x_0)\to \pi_{1}^{qtop}(X_j,x_0)$ which, together, induce a
continuous homomorphism 
\begin{equation*}
\phi:\pi_{1}^{qtop}(X,x_0)\to
\varprojlim\left(\pi_{1}^{qtop}(X_j,x_0),(r_{j+1,j})_{\ast}\right)
\end{equation*}
to the limit of the inverse system of quasitopological groups $\pi_{1}^{qtop}(X_j,x_0)$ and bonding maps $(r_{j+1,j})_{\ast}:
\pi_{1}^{qtop}(X_{j+1},x_0)\to \pi_{1}^{qtop}(X_j,x_0)$.

\begin{lemma}\label{phiinjective}
If $X=\varprojlim (X_j,r_{j+1,j})$ is an inverse limit of nested retracts and $\piqtopx$ is $T_1$, then $\phi:\piqtopx\to
\varprojlim\left(\pi_{1}^{qtop}(X_j,x_0),(r_{j+1,j})_{\ast}\right)$ is
injective.
\end{lemma}
\begin{proof}
Since the based loop space functor preserves limits, there is a canonical homeomorphism $\Omega(X,x_0)\cong\varprojlim \left(\Omega(X_j,x_0),\Omega(r_{j+1,j})\right)$ where the limit has the usual inverse limit topology (i.e. as a subspace of $\prod_{j}\Omega(X_j,x_0)$). Thus a loop $f\in \Omega(X,x_0)$ is identified with the sequence $(f_1,f_2,...)$ where $f_j=r_j\circ f$.

Suppose $f:\ui\to X$ is a loop such that $[f]\in \ker \phi$ and note $\phi([f])=([f_1],[f_2],...)$ is the identity, $[f_j]$ is the identity in $\pi_{1}^{qtop}(X_j,x_0)$ for each $j$. Let $g_j=s_j\circ f_j:\ui\to X_j\to X$. Since $f_j:\ui\to X$ is null-homotopic, so is $g_j$. Note that $g_j=(f_1,f_2,...,f_j,s_{j+1,j}\circ f_j,s_{j+2,j}\circ f_j,...)$.

Since we identify $\Omega(X,x_0)$ as subspace of the direct product $\prod_{j}\Omega(X_j,x_0)$, we have $g_j\to f$ in $\Omega(X,x_0)$. Since $\pi:\Omega(X,x_0)\to \piqtopx$ is continuous, $[g_j]\to [f]$ in $\piqtopx$ where each $[g_j]$ is the identity. But if $\piqtopx$ is $T_1$, every constant sequence has a unique limit and thus $[f]$ must be the identity of $\piqtopx$.
\end{proof}
\begin{theorem}\label{nestedretracts}
Suppose $X=\varprojlim (X_j,r_{j+1,j})$ is an inverse limit of nested retracts such that $\pi_{1}^{qtop}(X_j,x_0)$ is invariantly separated for each $j$. Then the following are equivalent:
\begin{enumerate}
\item $\piqtopx$ is $T_1$,
\item $\piqtopx$ is $T_2$,
\item $\piqtopx$ is totally separated,
\item $\piqtopx$ is invariantly separated.
\end{enumerate}
\end{theorem}

\begin{proof}
4. $\Rightarrow$ 3. $\Rightarrow$ 2. $\Rightarrow$ 1. follows from previous observations and basic topological facts. 1. $\Rightarrow$ 4. By assumption, $\pi_{1}^{qtop}(X_j,x_0)$ is invariantly separated for each $j$ (see Theorem \ref{discrete}). Therefore $\prod_j \pi_{1}^{qtop}(X_j,x_0)$ is invariantly separated (Proposition \ref{productinvar}) as is the sub-quasitopological group $\varprojlim\left(\pi_{1}^{qtop}(X_j,x_0),(r_{j+1,j})_{\ast}\right)$ (Proposition \ref{injection}). Since $\piqtopx$ is $T_1$, it follows from Lemma \ref{phiinjective} that the continuous homomorphism $\phi:\piqtopx\to
\varprojlim\left(\pi_{1}^{qtop}(X_j,x_0),(r_{j+1,j})_{\ast}\right)$ is injective. Apply Proposition \ref{injection} to see that $\piqtopx$ is invariantly separated.
\end{proof}
\begin{corollary}\label{nestedpolyhedra}
Suppose $X=\varprojlim (X_j,r_{j+1,j})$ is a locally path connected inverse limit of nested retracts where each $X_j$ is a locally path connected and semilocally simply connected metric space. Then $\piqtopx$ is $T_1$ iff $X$ is $\pi_1$-shape injective.
\end{corollary}
\begin{proof}
One direction is clear from Corollary \ref{injectivetohausdorff}. Since $X_j$ is locally path connected and semilocally simply connected, $\pi_{1}^{qtop}(X_j,x_0)$ is discrete (for each $j$) by Theorem \ref{discrete}. Thus if $\piqtopx$ is $T_1$, then $\piqtopx$ is invariantly separated by Theorem \ref{nestedretracts}. Theorem \ref{injectiveisinvsep} now gives that $X$ is $\pi_1$-shape injective.
\end{proof}
\section{On higher separation axioms}

Though every $T_2$ topological group is Tychonoff, it is not even true that every $T_2$ topological group is $T_4$ (normal and $T_2$) \cite{AT08}.

\begin{theorem}\label{t3tot4}
If $X$ is a compact metric space and $\pi_{1}^{qtop}(X,x_0)$ is a regular space, then $\pi_{1}^{qtop}(X,x_0)$ is normal. In particular, if $\piqtopx$ is $T_3$, then $\piqtopx$ is $T_4$.
\end{theorem}
\begin{proof}
Let $e$ denote the identity of $G=\piqtopx$. Recall that $G/\overline{e}$ is a $T_1$ quasitopological group. Since $G$ is regular by assumption, $G/\overline{e}$ is regular by Lemma \ref{topropstokolmquotient}. In particular, $G/\overline{e}$ is $T_2$. If $X$ is compact metric, then $\Omega(X,x_0)$ is a separable metric space \cite[4.2.17 \& 4.2.18]{Eng89} and is therefore Lindel\"{o}f. Note $G/\overline{e}$ is Lindel\"{o}f since it is the quotient of $\Omega(X,x_0)$. Since every regular Lindel\"{o}f space is paracompact and every paracompact $T_2$ space is normal, $G/\overline{e}$ is normal. It follows from Lemma \ref{topropstokolmquotient} that $G$ is normal.\\
\indent The second statement of the theorem is the special case where $e=\overline{e}$.
\end{proof}

\begin{corollary}\label{t3tot4topgrp}
If $X$ is a compact metric space and $\piqtopx$ is a topological group, then $\piqtopx$ is normal. In particular, if $\piqtopx$ is $T_1$, then $\piqtopx$ is $T_4$.
\end{corollary}

\begin{proof}
If $G=\piqtopx$ is a topological group with identity $e$, then $G/\overline{e}$ is a Tychonoff topological group. Thus $G$ is regular by Lemma \ref{topropstokolmquotient} and is therefore normal by Theorem \ref{t3tot4}. The second statement of the corollary is the special case where $e=\overline{e}$.
\end{proof}
As mentioned in the introduction, $\piqtopx$ often fails to be a topological group. In this case, Corollary \ref{t3tot4topgrp} does not apply. We are left with the following fundamental open problem.
\begin{problem}\label{openq1}
\emph{
If $X$ is a Peano continuum such that $\piqtopx$ is $T_1$ must $\piqtopx$ be $T_4$? Must $\piqtopx$ be normal for every compact metric space $X$?
}
\end{problem}
Theorem \ref{t3tot4} reduces this problem to the problem of deciding whether or not $\piqtopx$ must be regular.
\section{When is $\pi_{1}^{qtop}(X,x_0)$ a topological group?}

The topological properties of a group $G$ (endowed with a topology) which allow one to conclude that group multiplication is continuous are well-studied, e.g. \cite{Bouziad}\cite{Bouziad2}\cite{Ellis1}\cite{Mont}. For instance, the celebrated Ellis Theorem \cite{Ellis1} implies that a locally compact $T_2$ quasitopological group is a topological group. We consider such properties when $G$ is a quasitopological fundamental group.

It is known that there are many examples of spaces $X$ for which $\pi_{1}^{qtop}(X,x_0)$ is not a topological group. The following examples illustrate the variety of spaces for which this phenomenon can occur.
\begin{example} 
\emph{ \cite{Fabhe}
The Hawaiian earring $\mathbb{H}$ described in Example \ref{he1} is a $\pi_1$-shape injective Peano continuum, however, $\piqtop(\mathbb{H},(0,0))$ fails to be a topological group.
}
\end{example}
The following pair of spaces illustrate how the success or failure of group multiplication can distinguish homotopy type when standard application of shape theory fails to do so.
\begin{example}\label{lasso}
\emph{For $n\geq 1$, let $D_n\subseteq \mathbb{R}^2$ be the circle of radius $\frac{1}{(n+1)^2}$ centered at $\left(1,\frac{1}{n}\right)$. Let $c_n$ be the point in $D_n$ nearest the origin and $L_n=[0,c_n]$ be the line segment connecting the origin to $c_n$. The \textbf{lasso space} \[\mathbb{L}=\left([0,1]\times \{0\}\right)\cup \left(\bigcup_{n\geq 1}L_n\cup D_n\right)\] is a compact and semilocally simply connected, but non-locally path connected, planar set (See Figure 2) such that $\piqtop(\mathbb{L},(0,0))$ is free on countable generators (as an abstract group) but is not a topological group \cite{Fabcgqtop}.\\
\indent We use the Lasso space to show that continuity of multiplication in quasitopological fundamental groups can distinguish shape equivalent spaces. Compare $\mathbb{L}$ with the following construction, which is equivalent to the main example in \cite{Fabe06}: Let $M_n$ be the line segment connecting the origin to $\left(1,\frac{1}{n}\right)$. Define the planar set
 \[Y=\left([0,1]\times \{0\}\right)\cup \left(\bigcup_{n\geq 1}M_n\right)\cup\left(\bigcup_{n\geq 1}\{1\}\times \left[\frac{1}{2n-1},\frac{1}{2n}\right]\right).\]
Note that $Y$ and $\mathbb{L}$ are both weakly homotopy equivalent and shape equivalent, however, arguments in \cite{Brazfretopgrp} may be used to show that $\piqtopy$ is a topological group (namely, the free Graev topological group on the one-point compactification of the natural numbers). Thus $\piqtop(\mathbb{L},(0,0))$ and $\piqtop(Y,(0,0))$ are not isomorphic quasitopological groups. It follows from the homotopy invariance of $\piqtop$ that $\mathbb{L}$ and $Y$ are not homotopy equivalent.
}
\end{example}
\begin{figure}[H]
\centering \includegraphics[height=1.5in]{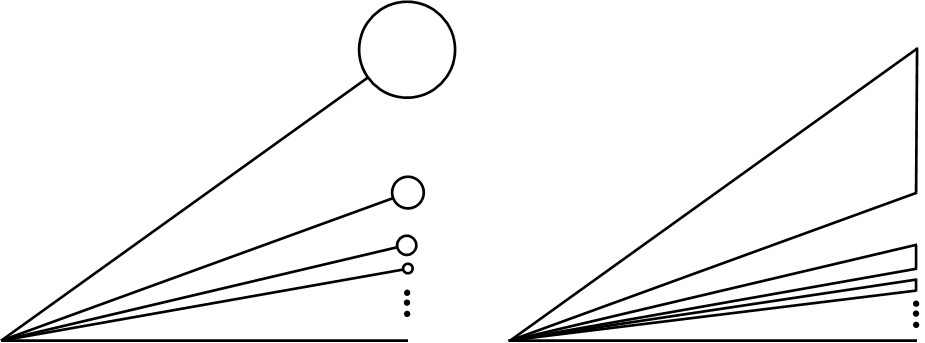}
\caption{The lasso space $\mathbb{L}$ (left) and $Y$ (right).}
\end{figure}
\begin{example}\label{hoop}
\emph{ \cite{Brazfretopgrp} For any subspace $A\subset [0,1]$, the \textbf{hoop space} (or \textbf{generalized wedge of circles}) on $A$ is the planar set
\begin{equation*}
hp(A)=\bigcup_{a\in A}\{(x,y)\in \mathbb{R}^2|(x-a)^2+y^2=(1+a)^2\}
\end{equation*}
If $A=\mathbb{Q}\cap (0,1)$, then $hp(A)$ is locally simply connected and has a fundamental group freely generated by a countable set (namely $\pi_0(A)=A$), however, $\pi_{1}^{qtop}(hp(A),(0,0))$ fails to be a topological group. On the other hand, if $A$ is compact and totally path disconnected (e.g. a cantor set or countable compact set), then $\pi_{1}^{qtop}(hp(A),(0,0))$ is a topological group (See Proposition \ref{komega} below).
}
\end{example}

Note $\piqtopx$ also fails to be a topological group if $X$ contains any of the above examples (of failure) as a retract.

The fact that $\piqtopx$ is not always topological group is intimately related to the fact that the product of quotient maps
\begin{equation*}
\pi\times \pi:\Omega(X,x_0)\times \Omega(X,x_0)\to
\pi_{1}^{qtop}(X,x_0)\times \pi_{1}^{qtop}(X,x_0)
\end{equation*}
can fail to be quotient. In the case that $\pi\times \pi$ is quotient, the continuity of multiplication in $\piqtopx$ follows from the universal property of quotient spaces.
\begin{lemma}\label{quotientstocontmult}
If $\pi\times \pi:\Omega(X,x_0)\times \Omega(X,x_0)\to
\piqtopx\times \piqtopx$ is a quotient map, then $\piqtopx$ is a topological group.
\end{lemma}
\begin{proof}
Suppose $\pi\times \pi$ is quotient. Since $\piqtopx$ is a quasitopological group it suffices to check that multiplication is jointly continuous. Consider the diagram
\[\xymatrix{ \Omega(X,x_0)\times \Omega(X,x_0) \ar[d]_{\pi\times \pi} \ar[r] &\Omega(X,x_0) \ar[d]^-{\pi}\\
\piqtopx\times \piqtopx \ar[r] & \piqtopx}\]
where the top and bottom horizontal maps are loop concatenation and group multiplication respectively. Since the top composition is continuous and $\pi\times \pi$ is quotient, group multiplication is continuous by the universal property of quotient spaces. Thus $\piqtopx$ is a topological group.
\end{proof}
\begin{lemma}\label{prodofquotientchar}
For any space $X$, the following are equivalent:
\begin{enumerate}
\item $\pi\times \pi:\Omega(X,x_0)\times \Omega(X,x_0)\to
\piqtopx\times \piqtopx$ is a quotient map.

\item The canonical group isomorphism $\rho:\piqtop(X\times
X,(x_0,x_0))\to\piqtopx\times \piqtopx$ is a
homeomorphism.

\item $\piqtop(X\times X,(x_0,x_0))$ is a topological group.
\end{enumerate}

If any of these conditions hold, then $\piqtopx$ is a
topological group.
\end{lemma}

\begin{proof}
1. $\Leftrightarrow$ 2. Consider the commuting diagram \[\xymatrix{ \Omega(X\times X,(x_0,x_0)) \ar[d]_{\pi} \ar[r]^-{\cong} &\Omega(X,x_0)\times \Omega(X,x_0) \ar[d]^-{\pi\times \pi}\\
\piqtop(X\times X,(x_0,x_0)) \ar[r]_-{\rho} & \piqtopx\times \piqtopx}\]where top map is the canonical homeomorphism and the bottom map is the canonical continuous group isomorphism $[(\alpha,\beta)]\mapsto ([\alpha],[\beta])$. Since the left vertical map is quotient, the right vertical map is quotient iff $\rho$ is a homeomorphism.\\
3. $\Rightarrow$ 2. Since $\rho$ is a continuous group isomorphism, it suffices to show the inverse is continuous. Let $\mu$ be the continuous multiplication map of $\piqtop(X\times X,(x_0,x_0))$. The inclusions $i,j:X\to X\times X$, $i(x)=(x,x_0)$ and $j(x)=(x_0,x)$ induce the continuous homomorphisms $i_{\ast},j_{\ast}:\piqtopx\to \piqtop(X\times X,(x_0,x_0))$ given by $i_{\ast}([\alpha])=[(\alpha,c_{x_0})]$ and $j_{\ast}([\beta])=[(c_{x_0},\beta)]$. The composition $\mu\circ (i_{\ast}\times j_{\ast})$ is continuous and is the inverse of $\rho$.\\
1. $\Rightarrow$ 3. If $\pi\times \pi$ is quotient, then $\piqtopx$ is a topological group by Lemma \ref{quotientstocontmult}. Since $\rho$ is a homeomorphism (by 1. $\Rightarrow$ 2.), $\piqtop(X\times X,(x_0,x_0))$ is isomorphic to the product of topological groups and is therefore a topological group.\\

\indent The last statement follows from the previous Lemma.
\end{proof}

The authors do not know of a space $X$ such that $\piqtopx$ is a topological group but for which $\pi\times \pi$ fails to be a quotient map.

The fact that multiplication in $\piqtopx$ may fail to be continuous is not a result of the definition of quotient topology but rather is due to the fact that the usual category of topological spaces is not Cartesian closed. A common approach to dealing with this categorical issue is to coreflect (in the functorial sense) to a Cartesian closed category of topological spaces \cite{Brown06}\cite{Strickland}.

\begin{definition}
\emph{A subset $A$ of a topological space $X$ is said to be \textbf{k-closed} if for each compact $T_2$ space $K$ and map $t:K\to X$, $t^{-1}(A)$ is closed in $K$. The space $X$ is \textbf{compactly generated}
if every k-closed set is closed in $X$. Let $\mathbf{CG}$ denote the
category of compactly generated spaces and continuous functions.}
\end{definition}

A space $X$ is compactly generated iff $X$ is the quotient space of a topological sum of compact $T_2$ spaces. For any space $X$, one can refine the topology of $X$ to contain all k-closed sets of $X$ and obtain a compactly generated space $\mathbf{k}X$. Note that $\mathbf{k}X=X$ iff $X$ is compactly generated. It is well-known that the category of compactly generated spaces
forms a convenient, Cartesian closed category of topological spaces when the k-product $X\times_{k}Y=\mathbf{k}(X\times Y)$ is used \cite{BT}\cite{Strickland}. Moreover, $\mathbf{CG}$ contains all sequential spaces (and thus all first countable spaces). We use the following Lemma, which is often considered to be an advantage of implementing the compactly generated category.

\begin{lemma}
\label{cgquotient}\cite{Strickland} The quotient of a compactly generated
space is compactly generated. Additionally, if $q:X\to Y$ and $q^{\prime
}:X^{\prime }\to Y^{\prime }$ are quotient maps of compactly generated
spaces, then $q\times_{k}q:X\times_{k}X\to Y\times_{k}Y^{\prime }$ is a
quotient map.
\end{lemma}

\begin{definition}
\emph{A $\mathbf{CG}$\textbf{-group} is a group object in $\mathbf{CG}$,
that is, a compactly generated group $G$ such that inverse $G\to G$ is
continuous and multiplication $G\times_k G\to G$ is continuous with respect
to the k-product.}
\end{definition}

\begin{lemma}
\label{prodofquot} If $q:X\to Y$ is a quotient map of spaces and $X\times X$
is compactly generated, then the product $q\times q:X\times X\to Y\times Y$
is a quotient map iff $Y\times Y$ is compactly generated.
\end{lemma}

\begin{proof}
One direction is obvious from the first statement of Lemma \ref{cgquotient}. Suppose $Y\times Y$ is compactly generated. Note both $X$ and $Y$ are compactly generated (as quotients of $X\times X$ and $Y\times Y$ respectively). Since the products $X\times X$ and $Y\times Y$ are compactly generated, the direct product and k-product topologies agree, i.e. $X\times X=X\times_{k}X$ and $Y\times Y=Y\times_{k}Y$. But $q\times_{k}q:X\times X\to Y\times_{k}Y'$ is quotient by Lemma \ref{cgquotient} and this is precisely $q\times q:X\times X\to Y\times Y$.
\end{proof}

\begin{theorem}
\label{cggroups} If $X$ is metrizable, then $\piqtopx$ is a $\mathbf{CG}$-group. Moreover, if $\piqtopx\times \piqtopx$ is compactly generated, then $\piqtopx$
is a topological group.
\end{theorem}

\begin{proof}
Recall that for a given metric on $X$, the compact-open topology of $\Omega(X,x_0)$ agrees with the topology induced by the uniform metric. Since $\Omega(X,x_0)$ is compactly generated, the quotient space $\piqtopx$ is compactly generated. The product $\Omega(X,x_0)\times \Omega(X,x_0)$ is metrizable and therefore compactly generated. Thus $\Omega(X,x_0)\times \Omega(X,x_0)=\Omega(X,x_0)\times_k \Omega(X,x_0)$. Consider the commuting diagram \[\xymatrix{
\Omega(X,x_0)\times \Omega(X,x_0) \ar[d]_-{\pi\times_{k}\pi} \ar[rr] &&\Omega(X,x_0) \ar[d]_-{\pi} \\
\piqtopx\times_{k}\piqtopx \ar[rr] && \piqtopx
}\]where the horizontal maps are loop concatenation and group multiplication. Since the $\pi\times_{k}\pi$ is quotient and the top composition is continuous, group multiplication is continuous by the universal property of quotient spaces. We have already established that inversion is continuous. Thus $\piqtopx$ is a $\cg$-group.

If the direct product $\piqtopx\times\piqtopx$ is compactly generated, then Lemma \ref{prodofquot} implies the direct product $\pi\times \pi$ is quotient. By Lemma \ref{quotientstocontmult}, $\piqtopx$ is a topological group.
\end{proof}

Since first countable spaces and their (finite) products are compactly generated, it follows that if $X$ is metrizable and $\piqtopx $ is first countable, then $\piqtopx$ is
a topological group. Moreover, it is well known that a first countable
topological groups is pseudometrizable (and metrizable if it is $T_0$).

\begin{corollary}
If $X$ is metrizable, then either $\piqtopx$ is a pseudometrizable topological group or $\piqtopx$ is not first countable.
\end{corollary}

\begin{example}\emph{
We obtain an example of a non-discrete metrizable fundamental group when $X=\prod_{n}X_n$ is an infinite product of (non-simply connected) CW-complexes $X_n$ (e.g. $X_n=S^1$ for each $n$). Since $\piqtop(X_n,x_n)$ is discrete, $\pi:\Omega(X_n,x_n)\to \piqtop(X_n,x_n)$ is an open map and therefore $\prod_{n}\pi:\prod_{n}\Omega(X_n,x_n)\to \prod_{n}\piqtop(X_n,x_n)$ is open. Consider the following diagram 
\[\xymatrix{
\Omega(X,x_0) \ar[d]_-{\pi} \ar[r]^-{\cong}  & \prod_{n}\Omega(X_n,x_n) \ar[d]^-{\prod_{n}\pi}\\
\piqtopx \ar[r] &\prod_{n}\piqtop(X_n,x_n)
}\]
where the top horizontal map is the canonical homeomorphism and the bottom map is the canonical group isomorphism. Since both vertical maps are quotient, $\piqtopx\to  \prod_{n}\piqtop(X_n,x_n)$ is a homeomorphism. Thus $\piqtopx$ is isomorphic to an infinite product of discrete groups and is a metrizable topological group with a neighborhood base of open invariant subgroups at the identity.
}
\end{example}

A space $Y$ is a $k_{\omega}$\textbf{-space} if it is the inductive (or direct) limit of a
sequence $Y_1\subset Y_2\subset Y_3\subset...$ (called a $k_{\omega}$-decomposition) of compact subspaces $Y_n$. Equivalently, $\bigcup_{n}Y_n=Y$ has the weak topology with respect to the nested subsets $\{Y_n\}$. If each $Y_n$ is Hausdorff,
then $Y$ is the quotient of the topological sum $\coprod_{n}Y_n$ of compact
Hausdorff spaces and is therefore compactly generated. A $k_{\omega}$-group
is a quasitopological group whose underlying space is a $k_{\omega}$-space.
It is known that if $X$ and $Y$ have $k_{\omega}$-decomposition $\{X_n\}$
and $\{Y_n\}$, then $\{X_n\times Y_n\}$ is a $k_{\omega}$-decomposition for
the direct product $X\times Y$ (See the appendix of \cite{BHsubgroup}). Thus finite products of $k_{\omega}$-spaces are
compactly generated. In light of this fact, we apply Theorems \ref{cggroups} and \ref{t3tot4}.

\begin{proposition}\label{komega}
If $X$ is a metric space and $\piqtopx$ is a $T_2$ $k_{\omega}$-group, then $\piqtopx$ is a topological group. Moreover, if $X$ is compact, then $\piqtopx$ is $T_4$.
\end{proposition}

\begin{example}
\emph{\cite{Brazfretopgrp}
Let $C\subset [0,1]$ be the standard middle third Cantor set and $X=hp(C)$ be the hoop space on $C$ as constructed in Example \ref{hoop}. Then $\piqtop(hp(C),(0,0))$ is the free group $F(C)$ generated by the underlying set of $C$ and is the inductive limit of the compact subspaces $F_n(C)$ consisting
of reduced words $x_{1}^{\epsilon_1}x_{2}^{\epsilon_2}...x_{m}^{\epsilon_m}$, $x_i\in C$, $\epsilon_i=\pm 1$ where $m\leq n$ \cite{Brazfretopgrp}. In particular, $F_n(C)$ is
the quotient of the compact Hausdorff space $\coprod_{0\leq m\leq n}(C\sqcup
C^{-1})^{m}$. Thus $\piqtop(hp(C),(0,0))$ is a topological group by Proposition \ref{komega}. In fact, $\piqtop(hp(C),(0,0))$ is isomorphic to the free Markov topological group on $C$. The same argument can be applied to any zero-dimensional, compact space $C\subseteq [0,1]$. }
\end{example}

\end{document}